\documentclass{amsart}

\usepackage{multirow}
\usepackage{amsfonts}
\usepackage{graphicx}
\usepackage{amscd}
\usepackage{amsmath}
\usepackage{amssymb}
\usepackage[matrix,arrow]{xy}
\usepackage[colorlinks=true, allcolors=blue]{hyperref}
\usepackage[usenames]{color}

\newtheorem{thm}{theorem}[section]
\newtheorem{theorem}[thm]{Theorem}
\newtheorem{proposition}[thm]{Proposition}
\newtheorem{lemma}[thm]{Lemma}
\newtheorem{corollary}[thm]{Corollary}
\newtheorem{remark}[thm]{Remark}

\begin{document}

\title[$\mathbb{Z}$-graded identities of the Lie algebras $U_1$ in characteristic 2]{$\mathbb{Z}$-graded identities of the Lie algebras $U_1$ in characteristic 2}
\author[Fidelis]{Claudemir Fidelis}
\thanks{C. Fidelis was supported by FAPESP grant No.~2019/12498-0}
\address{Unidade Acad\^emica de Matem\'atica, Universidade Federal de Campina Grande, Campina Grande, PB, 58429-970, Brazil
	\\ and \\Instituto de Matem\'atica e Estat\'istica da Universidade de S\~ao Paulo, SP, 05508-090, Brazil
}
\email{claudemir@mat.ufcg.edu.br}

\author[Koshlukov]{Plamen Koshlukov}
\thanks{P. Koshlukov was partially supported by FAPESP grant No.~2018/23690-6 and by CNPq grant No.~302238/2019-0}
\address{Department of Mathematics, UNICAMP, 13083-859 Campinas, SP,  Brazil}
\email{plamen@unicamp.br}

\begin{abstract}
Let $K$ be any field of characteristic two and let $U_1$ and $W_1$ be the Lie algebras of the derivations of the algebra of Laurent polynomials $K[t,t^{-1}]$ and of the polynomial ring $K[t]$, respectively. The algebras $U_1$ and $W_1$ are equipped with natural $\mathbb{Z}$-gradings. In this paper, we provide bases for the graded identities of $U_1$ and $W_1$, and we prove that they do not admit any finite basis.
\medskip


\medskip

\end{abstract}
\maketitle

\section{Introduction}

The description of the polynomial identities satisfied by an algebra depends heavily on the base field. If the field is of characteristic 0, one may consider multilinear polynomial identities since they determine all identities of a given algebra. One of the main tools in this case is  the representation theory of the symmetric and of the general linear groups, and refinements, see for a detailed account the monographs \cite{drbook, ZG}. When the field $K$ is infinite, one has to consider multihomogeneous identities, see for example \cite[Section 4.2]{drbook}. The methods one uses in this case are mostly based on Invariant theory \cite{CC}. Finally, if $K$ is a finite field then neither of the above types of identities is sufficient. As a rule, neither of the methods described works properly. Instead one uses structure theory  \cite{KS,kruse,lvov}, together with combinatorics based on the properties of finite fields \cite{OatesP}.

In spite of the extensive research in this area little is known about the concrete form of the identities satisfied by a given algebra. The monographs \cite{drbook, ZG}, and their references, give a good account on the results already obtained. But it should be stressed that the list of algebras whose identities are known is very short and easy to reproduce. In the associative case one knows the identities of the matrix algebras $M_2(K)$ for infinite fields of characteristic different from 2 (see for example \cite{razmbook,kjam2}).
The identities of the Grassmann algebra $E$ are also known, see for example \cite[Section 5.1]{drbook}, as well as those of its tensor square $E\otimes E$ \cite{popov}. Add here the upper triangular matrices of any order $n$, $UT_n(K)$, and that is it. In the case of Lie algebras one knows the identities of $UT_n(K)$ (as a Lie algebra), and also of $sl_2(K)$ over infinite fields \cite{razmbook, vasal}. We only mention that even the identities of $M_3(K)$, over a field of characteristic 0, are not known. A ``nagging'' and long standing problem is to determine whether the identities of $M_2(K)$ admit a finite basis whenever $K$ is an infinite field of characteristic 2. 

The theory developed by A. Kemer in the eighties \cite{basekemer} led him to the positive solution of the famous Specht problem: Is every ideal of identities of an associative algebra in characteristic 0 finitely generated as an ideal of identities? Kemer's results are far from constructive and do not provide concrete bases of the identities of a given algebra. 

We have outlined above some of the reasons that led researchers to look for other types of polynomial identities. These include identities with involution, weak identities, group graded identities. We shall not comment on the former two types of identities as we work exclusively with the latter type. Graded identities appeared in Kemer's research; his methods rely heavily on associative superalgebras (that is $\mathbb{Z}_2$-graded algebras). Later on gradings on algebras and their graded identities became an important part of PI theory. The gradings on matrix algebras were described in \cite{bsz}, see also \cite{bz}, assuming the base field is algebraically closed.

The matrix algebra $M_n(K)$ admits natural gradings by the groups $\mathbb{Z}_n$ and $\mathbb{Z}$. The corresponding graded identities were described in \cite{vasca,vaspams} in characteristic 0, and in \cite{ssaca, ssaserd} in positive characteristic. An extensive research on gradings and graded identities for classes of important algebras has been conducted, we direct the interested reader to \cite{pkfyja} for further information and references. 

The simple finite dimensional Lie algebras over an algebraically closed field of characteristic 0 are well known. In the infinite dimensional case the so-called algebras of Cartan type appear. We denote by $W_1$ the Lie algebra of derivations of the polynomial ring in one variable $K[t]$, and by $U_1$ the algebra of derivations of the Laurent polynomials $K[t,t^{-1}]$. The former is known as the Witt algebra, the latter gives rise to the Virasoro algebra. Both algebras have canonical gradings by the group $\mathbb{Z}$ such that every non-zero component is one dimensional. The algebras $W_1$ and $U_1$ appear naturally in various branches of Physics and Mathematics, the interested reader can consult the paper \cite{Huang} and its references for an extensive treatment of this topic. 

The $n$-variable analogues of $W_1$ and $U_1$, the algebras $W_n$ and $U_n$ are defined as the derivations of the corresponding polynomial and Laurent polynomial rings in $n$ variables. These were first studied around 1910 by E. Cartan in his classification of simple Lie algebras in characteristic 0. Later on it was discovered that these algebras, although not simple in positive characteristic, produce naturally various simple finite dimensional Lie algebras as their homomorphic images. The celebrated theorems of V. Kac \cite{kac, kacbook} classify the simple Lie algebras graded by $\mathbb{Z}$ under some natural conditions. Namely the algebra $L=\oplus_{i\in \mathbb{Z}} L_i$ must be of polynomial growth: $\sum_{j\le i} \dim L_j$ grows like a polynomial in $i$; $L_0$ acts irreducibly on $L_{-1}$, and $L$ is generated by degrees 0 and $\pm 1$.  O.  Mathieu \cite{olmat} classified the simple $\mathbb{Z}$-graded Lie algebras of polynomial growth. The algebras $W_n$ and $U_n$ were also studied by Kaplansky \cite{kapl}.

The graded identities for the algebra $W_1$ were described in \cite{FKK}, assuming that the field is of characteristic 0. The ones for $U_1$, with its canonical grading by the group $\mathbb{Z}$ were recently obtained in \cite{CP}, over an infinite field of characteristic different from two. As  a consequence, the main result of \cite{FKK} was generalized. 

In this paper we study the graded identities satisfied by the Lie algebra $U_1$, equipped with its natural $\mathbb{Z}$-grading, over an arbitrary field of characteristic two. We produce a basis of the graded identities for $U_1$. As a consequence we obtain a basis of those for $W_1$ as well. Furthermore, we prove that the graded identities of $U_1$, as well as these of $W_1$, do not admit any finite basis. The counterparts of these results over an infinite field of characteristic different from two, were obtained in \cite{CP,FKK}.

The ordinary identities of $W_1$ coincide with the identities of the Lie algebra of the vector fields on the line if $K=\mathbb{R}$ is the real field. The standard Lie polynomial of order $4$ (which is of degree 5 and is alternating in 4 of its variables) is an identity of $W_1$. On the other hand, it is a long-standing open problem to determine a basis of the identities satisfied by $W_1$. The vector space of $W_n$, the derivations of the polynomial ring in $n$ variables, can be given the structure of a left-symmetric algebra, denoted by $L_n$. In \cite{koum} the authors studied the right-operator identities of $L_n$, and described a large class of general identities for $L_n$.  We hope that our results about the $\mathbb{Z}$-graded identities of $U_1$ may shed additional light on the polynomial identities satisfied by $W_1$, and consequently by $U_1$.

\section{Definitions and preliminary results} \label{Preliminaries}

We fix a field $K$, all algebras and vector spaces we consider will be over $K$. If $A$ is an associative algebra one defines on the vector space of $A$ the Lie bracket $[a,b]=ab-ba$. Denote by $A^{(-)}$ the Lie algebra thus obtained, the Poincar\'e--Birkhoff--Witt theorem yields that every Lie algebra is a subalgebra of some $A^{(-)}$.

Let $L$ be an algebra (not necessarily associative) and let $G$ be a group. A $G$-grading on $L$ is a vector space decomposition
\begin{align}\label{gr}
\Gamma\colon L=\oplus_{g\in G}L_g
\end{align}
such that $L_gL_h\subseteq L_{gh}$, for all $g$, $h\in G$. In this case one says that $L$ is $G$-graded. The subspaces $L_g$ are the homogeneous components of the grading and a non-zero element
$a$ of $L$ is homogeneous if $a\in L_g$ for some $g\in G$; we denote this by $\|a\|_G=g$ (or simply $\|a\|=g$ when the group $G$ is clear from the context). The support of the grading is the set $\mathrm{supp}\ L=\{g\in G \mid L_g\neq 0\}$. A subalgebra (an ideal, a subspace) $B$ of $A$ is a graded subalgebra (respectively ideal, subspace) if $B=\oplus_{g\in G} A_{g}\cap B$. 

The first example is the Witt algebra $W_1=Der(K[t])$. It is the Lie algebras of the derivations of the polynomial ring $K[t]$. The elements $e_n=t^{n+1}d/dt$, $n\geq -1$, form a basis of $W_1$. The Lie algebra structure on the vector space $W_1$ is given by the multiplication
\begin{equation}\label{multiwitt}
[e_i, e_j] = (j-i)e_{i+j}.
\end{equation}
The algebra $W_1$ has a $\mathbb{Z}$-grading,
$W_1=\oplus_{i\in\mathbb{Z}}L_i$, where $L_i=0$ whenever $i\leq -2$, and $L_i$ is the (one-dimensional) span of $e_i$ if $i\geq -1$. Thus the element $e_n$ is homogeneous of degree $n$. 

Another example of graded algebras is the algebra $U_1$. Let $A=K[t,t^{-1}]$ be the algebra of Laurent polynomials in one variable $t$. Then $U_1$ is the Lie algebras of the derivations of $A$. The elements $e_n=t^{n+1}d/dt$, $n\in\mathbb{Z}$, form a basis of $U_1$; the multiplication in $U_1$ is also given by \eqref{multiwitt}. The algebra $U_1$ is $\mathbb{Z}$-graded, $U_1=\bigoplus_{i\in\mathbb{Z}}L_i$ where $L_i$ is the span of $e_i$, for each $i\in\mathbb{Z}$. This means that $U_1$ has full support on $\mathbb{Z}$, in other words $\mathrm{supp}\, U_1= \mathbb{Z}$.

Since we shall work with the above two graded algebras we will refrain from giving other examples. 

Let $ X=\cup_{i\in\mathbb{Z}}X_{i}$ be the disjoint union of infinite countable sets of variables $X_{i}=\{x_{1}^{i},x_{2}^{i},\ldots\}$, $i\in\mathbb{Z}$. Assuming that for each $i\in\mathbb{Z}$ the elements of the set $X_{i}$ are of $\mathbb{Z}$-degree $i$, the free associative algebra $K\langle X_\mathbb{Z}\rangle$ has a natural $\mathbb{Z}$-grading $\oplus_{i\in\mathbb{Z}} K\langle X_\mathbb{Z}\rangle^{i}$. Here $K\langle X_\mathbb{Z}\rangle^{i}$ is the vector subspace of $K\langle X_\mathbb{Z}\rangle$ spanned by all monomials of $\mathbb{Z}$-degree $i$. The subalgebra $L\langle X_\mathbb{Z} \rangle$ of $K\langle X_\mathbb{Z}\rangle^{(-)}$ generated by the set $X_{\mathbb{Z}}$ is the free $\mathbb{Z}$-graded Lie algebra, freely generated by $X_\mathbb{Z}$. Note that $L\langle X_\mathbb{Z} \rangle$ is a graded subspace of $K\langle X_\mathbb{Z} \rangle$ and that the corresponding decomposition gives a $\mathbb{Z}$-grading on $L\langle X_\mathbb{Z} \rangle$. The elements of $L\langle X_\mathbb{Z} \rangle$ are called $\mathbb{Z}$-graded polynomials (or simply polynomials). The degree of a polynomial $f$ in $x_i^{a_i}$, denoted by $\deg_{x_i^{a_i}}f$, counts how many times the variable $x_i^{a_i}$ appears in the monomials of $f$, and it is defined in the usual way. The definitions of multilinear and multihomogeneous polynomials are the natural ones. We define the (left-normed) commutator $[l_1,\cdots, l_n]$ of $n\geq 2$ elements $l_1$, \dots, $l_n$ in a Lie algebra $L$ inductively,  $[l_1,\cdots, l_n]=[[l_1,\cdots,l_{n-1}],l_n]$ for $n>2$.

Let $L=\oplus_{i\in \mathbb{Z}}L_i$ be a Lie algebra with a $\mathbb{Z}$-grading. An admissible substitution for the polynomial $f(x_{1}^{a_1},\ldots, x_{n}^{a_n})$ in $L$ is an $n$-tuple $(l_1,\ldots, l_n)\in L^{n}$ such that $l_i\in L_{a_i}$, for $i=1$, \dots, $n$. If $f(l_1,\dots,l_n)=0$ for every admissible substitution $(l_1,\dots, l_n)$ we say that $f(x_{1}^{a_1},\dots, x_{n}^{a_n})$ is a graded identity for $L$. The set of $\mathbb{Z}$-graded polynomial identities of $L$ will be denoted by $T_\mathbb{Z}(L)$. It is a $T_\mathbb{Z}$-ideal, that is an ideal invariant under the endomorphisms of $L\langle X_\mathbb{Z} \rangle$ as a graded algebra. The intersection of a family of $T_\mathbb{Z}$-ideals in $L\langle X_\mathbb{Z} \rangle$ is a $T_\mathbb{Z}$-ideal; given a set of polynomials $S\subseteq L\langle X_\mathbb{Z}\rangle$ we denote by $\langle S \rangle_\mathbb{Z}$ the intersection of the $T_\mathbb{Z}$-ideals of $L\langle X_\mathbb{Z} \rangle$ that contain $S$. We call $\langle S \rangle_\mathbb{Z}$ the $T_\mathbb{Z}$-ideal generated by $S$, and refer to $S$ as a basis of this $T_\mathbb{Z}$-ideal. It is well known that in characteristic 0, every $T_\mathbb{Z}$-ideal $T_\mathbb{Z}(L)$ is generated by its multilinear polynomials. Over an infinite field of positive characteristic one has to take into account the multihomogeneous polynomials instead of the multilinear ones.

In this paper, unless otherwise stated, $K$ denotes a field of characteristic two. We do not require any further restrictions on $K$. Our main result is the following theorem which provides a basis of the $\mathbb{Z}$-graded identities for $U_1$.

\begin{theorem}\label{mainresult1}
	Let $K$ be a field of characteristic two. The ideal of the graded identities of $U_1$ is generated, as a $T_\mathbb{Z}$-ideal, by the polynomials
	\begin{equation}\label{Cbasis1}
	[x_1^a,x_2^b] \equiv 0,
	\end{equation}
	where $a$ and $b$ are integers of the same parity, that is $a\equiv b\pmod{2}$.
\end{theorem}

We deduce, for the graded identities of $W_1$, the following theorem.

\begin{theorem}\label{ideW1}
	Let $K$ be a field of characteristic two. The $\mathbb{Z}$-graded identities
	\[
	x^c\equiv 0\quad(c\leq -2), \qquad [x_1^a,x_2^b] \equiv 0,
	\]
	where $a$ and $b$ are integers greater than $-2$, and of the same parity, form a basis for the $\mathbb{Z}$-graded identities of the Lie algebra $W_1$ over $K$.
\end{theorem}

\section{$\mathbb{Z}$-Graded identities of $U_1$} \label{identitiesinfinite}

Here we prove Theorem \ref{mainresult1}. To this end, we need a series of results.

\begin{lemma}\label{3.6}
Over a field of characteristic two, the graded identities \eqref{Cbasis1} hold for $U_1$.
\end{lemma}
\begin{proof}
	The proof is immediate by the multiplication rules  \eqref{multiwitt}.
\end{proof}

In analogy with the associative case we will call \textsl{monomials} the left normed commutators in the free graded Lie algebra.

\begin{proposition}\label{monident}
	Over a field of characteristic two, every graded monomial identity of $U_1$ is consequence of the identities \eqref{Cbasis1}.
\end{proposition}
\begin{proof}
	We denote by $I$ the $T_G$-ideal generated by the polynomials \eqref{Cbasis1}. We prove the claim by induction on the length $n$ of the monomial. The result is obvious for $n=1$ and $n=2$, so we suppose $n\geq 3$. Let $M^{\prime}=[x_{1}^{a_1},\ldots, x_{n-1}^{a_{n-1}}]$. If $M^{\prime}\in T_\mathbb{Z}(U_1)$ then it lies in $I$, by the induction hypothesis, and hence $M\in I$. We assume now that $M^{\prime}\notin T_\mathbb{Z}(U_1)$. Let $a=\| M^\prime\|$, then the result of every admissible substitution in $M^{\prime}$ is a scalar multiple of $e_{a}$. 
	Therefore, $M\in T_\mathbb{Z}(U_1)$ if and only if $[x_{1}^{a},x_{2}^{a_n}]\in T_\mathbb{Z}(U_1)$. The commutator $[x_{1}^{a},x_{2}^{a_n}]$ lies in $T_\mathbb{Z}(U_1)$ if and only if $a$ and $a_n$ have the same parity, and hence $M$ lies in $I$. Thus in all cases $M\in I$, as required.
\end{proof}

It is well known that every $T_\mathbb{Z}$-ideal is generated by its regular polynomials. Recall that a polynomial $f\in L\langle X_\mathbb{Z}\rangle$ is regular if every one of its variables appears in every monomial of $f$, not necessarily with the same degree. 

Recall that $I$ is the $T_\mathbb{Z}$-ideal generated by the polynomials in \eqref{Cbasis1}. The next lemma is a key step in the proof of our main theorem. 

\begin{lemma}\label{1111}
	Let $M=[x_{i_0}^{a_0},x_{i_1}^{a_1},\ldots, x_{i_n}^{a_n}]$ be a monomial in $L\langle X_\mathbb{Z} \rangle$. If $M\notin T_\mathbb{Z}(U_1)$ then $M$, modulo $I$, can be written in such a way that $a_0$ is odd, $a_i$ is even for each $i=1$, \ldots, $n$, and $a_1\leq a_2\leq \ldots\leq a_n$.
\end{lemma}
\begin{proof}
Since $M\notin T_\mathbb{Z}(U_1)$ either $a_0$ or $a_1$ is an odd integer. 
Suppose there exist $i$ and $j$, with $0\leq i< j\leq n$, such that $ a_i $ and $ a_j $ are odd integers. 
By exchanging the first two variables, if necessary, we can assume that $i = 0$ and $j = n$, and all other $a_k$ are even integers. (If some additional $a_k$ is odd we simply consider the initial part of the commutator, from $a_0$ to $a_k$.)  Then $M^\prime=[x_{i_0}^{a_0}, x_{i_1}^{a_1}, \ldots, x_{i_{n-1}}^{a_{n-1}}]$ is a monomial in $L\langle X_\mathbb{Z} \rangle$ that is not a graded identity. This implies $\sum_{i=0}^{n-1}a_i$ is odd, that is $M\in T_\mathbb{Z}(U_1)$ which is absurd. Now we apply identity \eqref{Cbasis1}, together with the Jacobi identity,
several times on the variables of even $\mathbb{Z}$-degree.
\end{proof}

As mentioned earlier, we cannot claim in general that a $T_\mathbb{Z}$-ideal is generated by its multihomogeneous elements; this is certainly true if the base field is infinite. But, in our specific case, we are in a position to deduce this fact.

\begin{lemma}\label{homogeneo}
Over a field of characteristic two, the $T_\mathbb{Z}$-ideal $T_\mathbb{Z}(U_1)$ is generated by its multihomogeneous polynomials.
\end{lemma}
\begin{proof}
The result is well known if $K$ is infinite. So we assume $K$ is a finite field. Let $f=f(x_0,x_1,\ldots,x_n)$ be a regular polynomial in $T_\mathbb{Z}(U_1)$. By Lemma \ref{1111}, only one of the variables in $f$ is in an odd component. Suppose $x_0$ in an odd component and the remaining variables are in even components. Moreover, $f$ is linear in $x_0$, and $x_0$ appears in the first position on each monomial in $f$. We  induct on the degree of the polynomial $f$. Suppose that $f$ is not multihomogeneous. As $f$ is regular, we write $f=\sum_i\alpha_iM_i$ where $M_i$ is the commutator 
\[
[x_0,x_1,\cdots,x_1,x_2,\ldots,x_{k-1},\underbrace{x_{k},\ldots,x_{k}}_{m^k_i\mbox{ times}},x_{k+1},\ldots,x_n] = x_0 (ad\,x_1)^{m_i^1}\cdots (ad\,x_n)^{m_i^n}.
\] 
Here $m^k_i$ are non-negative integers depending on the variable $x_i$ and on the monomial $M_i$, and $ad\,x$ is, as usual, the linear transformation $y\, ad\,x \mapsto [y,x]$ in a Lie algebra. For each variable $x_k$ we put $r_k=\min\{m_{i}^k\}$.
By the regularity of $f$ we have $r_k>0$. Applying  identity \eqref{Cbasis1}, if needed, there exists an element $f'$ which is a sum of regular polynomials:
$f^\prime=f^\prime(x_0,x_1,\ldots,x_n)$ in $T_\mathbb{Z}(U_1)$, not necessarily having the same variables as $f$, such that
\[
f\equiv f^\prime (ad\, x_1)^{r_1} (ad\, x_2)^{r_2} \cdots (ad\, x_n)^{r_n} \pmod{I}.
\]
	Note that the degree of the polynomial $f^\prime$ is lower than the degree of $f$. Since $I\subseteq T_{\mathbb{Z}}(U_1)$, by induction the result follows if $f^\prime\in T_\mathbb{Z}(U_1)$. Suppose that $f^\prime\notin T_\mathbb{Z}(U_1)$. Then there exists an admissible substitution for the polynomial $f$, say $(l_0,l_1,\ldots, l_n)$, of elements in $U_1$ such that  $f^\prime(l_0,l_1,\ldots,l_n)\neq 0$, but $f(l_0,l_1,\ldots,l_n)=0$. Recall that $U_1=\oplus_{i\in\mathbb{Z}}L_i$ is  the natural $\mathbb{Z}$-grading on $U_1$. Define $A_0$ as the sum of all even components $L_i$, $i\equiv 0\pmod{2}$, and $A_1$ as the sum of the odd ones. Of course, $f^\prime(l_0,l_1,\dots,l_n)\in A_1$, that is 
\[
f^\prime(l_0,l_1,\dots,l_n)=\sum_{i\in\mathbb{Z}}\lambda_ie_{2i+1}
\]
where the set $\{i\mid \lambda_i\neq 0\}$ is finite. As each homogeneous component of $U_1$ is one-dimensional, we have 
\[
0=f(l_0,l_1,\dots,l_n)= f^\prime(l_0,l_1,\dots,l_n) (ad\, e_{a_1})^{r_1} \cdots (ad\, e_{a_n})^{r_n} =\sum_{i\in \mathbb{Z}}\alpha_i \lambda_i e_{(2i+1)+t}.
\]
Here $t=\sum_{j=1}^{n}r_ja_j$. It is clear that each $\alpha_i$ equals 1, this implies that each $\lambda_i=0$, which contradicts the fact that the substitution in $f^\prime$ is not zero. Therefore the graded identity $f$ is equivalent to a multihomogeneous graded identity.
\end{proof}

The corollary to the above lemma is the key for the proof of our main theorem. When $K$ is an infinite field, its proof can be found in \cite[Theorem 4.4]{CP}. Formally in \cite{CP} it was required that $K$ an infinite field of characteristic different from 2, but the proof in characteristic 2, even for finite fields, remains essentially the same. 

\begin{corollary}\label{identimult}
	Let $K$ be a field of characteristic two. The $T_\mathbb{Z}$-ideal $T_\mathbb{Z}(U_1)$ is generated by its multilinear 
polynomials.
\end{corollary}
\begin{proof}
	By Lemma~\ref{homogeneo}, we can consider the multihomogeneous graded identities. Let $f=f(x_{1}^{a_1},\ldots, x_{r}^{a_r})$ be a multihomogeneous graded identity of $U_1$. As $\dim_KL_i\leq 1$  for every $i\in \mathbb{Z}$, we have that each admissible substitution $\varphi$ can be taken to map $x_{i}^{a_i}$ to $\xi_ie_{a_i}$ where
	$\xi$'s are commutative and associative (independent) variables over $K$. Here, as above, $e_{a_i}$ spans the vector space  $L_{a_i}$. Hence 
	we have
	\begin{equation}\label{idemult}
	\varphi(f(x_{1}^{a_1},\ldots,x_{r}^{a_r}))=\xi_1^{n_1}\cdots \xi_r^{n_r}f(e_{a_1},\ldots,e_{a_r}).
	\end{equation}
	Fix some $a_i$, $\deg_{x_i^{a_i}}f=n_i\geq 1$. Take new variables $x_{i,j}^{a_i}$ where $1\leq j\leq n_i$, and consider a multihomogeneous graded polynomial $h(x_{1,1}^{a_1},\ldots,x_{1,n_1}^{a_1},x_{2,1}^{a_2},\ldots, x_{r,n_r}^{a_r})$ such that (here we put, in order to simplify the notation, $i=1$)
	\[
	h(\underbrace{x_{1}^{a_1},\ldots,x_{1}^{a_1}}_{n_1},x_{2}^{a_2},\ldots, x_{r}^{a_r})=f(x_{1}^{a_1},\ldots,x_{r}^{a_r}).
	\]
	But $x_{i,j}^{a_i}$ can be evaluated to $e_{a_i}$ for each $1\leq i\leq r$ and $1\leq j\leq n_i$. Then Eq.~\eqref{idemult} implies that $h$ is a graded identity and that $f$ is a consequence of $h$. Moreover $h$ is linear in each of the new $n_i$ variables. Continuing in this way for the remaining variables we prove the statement.
\end{proof}

\begin{remark}
The choice of the polynomial $h$ in the above theorem need not be given by the multilinearisation process. Due to this reason it need not be unique. For example $f = [x^g_1, x^h_2, x^g_1]$ can come from $g = [x^g_1, x^h_2, x^g_3]$. The complete linearisation of $f$ is $[x^g_1, x^h_2, x^g_3] + [x^g_3, x^h_2, x^g_1]$. Clearly in characteristic 2 we cannot return to $f$ starting with the latter polynomial.
\end{remark}

Let $n$ be a positive integer. Given ${\bf g}\in \mathbb{Z}^n$ and a subset $\mathfrak{S}=\{s_1,\dots,s_{n}\}\subset \mathbb{N}$, with $s_1<\ldots<s_{n}$, denote by $P_{\mathfrak{S}}^{\bf g}$ the subspace of $L\langle X_\mathbb{Z} \rangle$ of the multilinear polynomials in the variables $\{x_{s_1}^{ a_1},\ldots, x_{s_{n}}^{a_n}\}$. The elements of $\mathbb{Z}$ and the $n$-tuple ${\bf g}$ may be omitted from the notation if no ambiguity arises, thus we write $P_{\mathfrak{S}}$ and $x_i$ instead of $P_\mathfrak{S}^{\bf g}$ and $x_{i}^{a_i}$, respectively. Let $\sigma$ be a permutation in $S_{n-1}$. Denote by $N_{\sigma}$ the monomial $[x_{s_n},x_{s_{\sigma(1)}},\cdots,x_{s_{\sigma(n-1)}}]$. It is well known that the set $\{N_\sigma \mid \sigma \in S_{n-1}\}$ is a basis for the vector space $P_\mathfrak{S}$.

Now we have all the ingredients for the proof of the main result in the paper.

\begin{proof}[Proof of Theorem \ref{mainresult1}]
	We have $I\subseteq T_\mathbb{Z}(U_1)$. To prove the opposite inclusion, it is enough to show that $P_n^{\bf g}\cap T_\mathbb{Z}(U_1)\subseteq I$ for every $n$ and every ${\bf g}\in \mathbb{Z}^n$. Let $f$ be an element in $P_n^{\bf g}\cap T_\mathbb{Z}(U_1)$. The cases $n=1$ and~2 are trivial. Hence we suppose $n\geq 3$. If $f=f(x_1,\ldots,x_n)$ has two variables in odd components, we have $f\in I$. Therefore, we assume that the variable $x_n$ of $f$ lies in an odd component and the remaining variables are in even components.  Let $f=\sum_{\sigma \in S_{n-1}} \lambda_\sigma N_{\sigma}$, $\lambda_{\sigma}\in {K}$, be a polynomial in $P_n^{\bf g}\cap T_\mathbb{Z}(U_1)$. Lemma~\ref{1111} implies that there exists $\lambda\in K$ such that
	\[
	f\equiv \lambda [x_n,x_1,\ldots,x_{n-1}] \pmod{I}
	\]
where $\|x_1\|<\ldots< \|x_{n-1}\|$. Since $f\in T_\mathbb{Z}(U_1)$, we have: 1) if $[x_n,x_1,\ldots,x_{n-1}]\notin T_\mathbb{Z}(U_1)$ it follows that $\lambda=0$, hence $f\in I$; 2) if $[x_n,x_1,\ldots,x_{n-1}]\in T_\mathbb{Z}(U_1)$ we apply Proposition \ref{monident}.
\end{proof}

The above results are easily adaptable to the case of the Lie algebra $W_1$. Recall that in \cite{FKK} the authors considered the base field of characteristic 0, and this was important in their proofs. In \cite{CP}, the result was extended to $W_1$ considered over an infinite field of characteristic different from $2$.

\section{Independence of graded identities}
In this section we use ideas from \cite{CP,FKK}. As above $K$ is an arbitrary field of characteristic two.

Denote by $f_{r,s}=[x^r_1, x^s_2] \in L\langle X_\mathbb{Z}\rangle$ the graded polynomials from \eqref{Cbasis1}, and assume $r\le s$. 

\begin{lemma}
Suppose $r$ and $s$ are integers of the same parity, $r\equiv s\pmod{2}$. The graded identity $f_{r,s}$ is not a consequence of the identities $f_{u,v}$, $u\le v$, if $(r,s)\neq (u,v)$.
\end{lemma}
\begin{proof}
Let $r$, $s\in\mathbb{Z}$ be of the same parity, 
and let $H=UT(3, K)$ be the Lie algebra of strictly upper triangular $3\times 3$ matrices over $K$. Define the vector subspaces $H_k$, $k\in\mathbb{Z}$, in $H$ as follows: 
\begin{enumerate}
	\item if $r\neq s$ we set $H_k=0$ for all $k\neq r$, $s$ and $r+s$; $H_r$ is the span of $E_{12}$, $H_s$ is the span of $E_{23}$ and $H_{r+s}$ is the span of $E_{13}$;
	\item if $r=s=0$ then $H_0=H$, and $H_k=0$ for every $k\neq 0$; 
	\item if $r=s\neq 0$ then $H_r$ is spanned by $E_{12}$ and $E_{23}$, $H_{2r}$ is spanned by $E_{13}$, and $H_k=0$ for every $k\ne r$, $2r$.
\end{enumerate}
Here $E_{ij}$ is the matrix that has $1$ at position $(i,j)$ and 0 elsewhere. It is clear that $H=\oplus_{i\in\mathbb{Z}}H_i$ is a $\mathbb{Z}$-graded Lie algebra.  Since $[E_{12},E_{23}]=E_{13}\neq0$, the graded identity $[x^r_1, x^s_2]$ is not satisfied in $H$. On the other hand, one can easily see that $H$ satisfies all graded identities \eqref{Cbasis1} as well as all identities $f_{u,v}$ 
when $(r,s)\neq (u,v)$. The result follows.
\end{proof}

A set $I$ of (graded) polynomials is an independent set of (graded) identities if neither of them lies in the ideal of (graded) identities generated by the remaining ones. 

\begin{corollary}
The set of polynomials $\{f_{r,s}\mid r,s\in\mathbb{Z}, r\le s, \quad r\equiv s\pmod{2}\}$ is an independent set of graded identities in $L\langle X_\mathbb{Z}\rangle$.
\end{corollary}

The above statements together with Theorem \ref{mainresult1} yield the following theorem.

\begin{theorem}\label{bseminimalU_1}
	Over a field of characteristic two, the graded identities 
	\[
	[x_1^a,x_2^b] \equiv 0, \quad a\le b
	\]
	where $a$ and $b$ are of the same parity, form a minimal basis for the $\mathbb{Z}$-graded identities of $U_1$.
\end{theorem}

For the Lie algebra $W_1$, we add to the list of identities the variables $x^c$ with $c< -1$. Note that the identity $f_{-1,-1}=[x^{-1}_1, x^{-1}_2]$ is a consequence of $x^{-2}$.

\begin{lemma}
For each $d\in\mathbb{Z}$, the graded identity $x^d$ is not a consequence of the graded identities \eqref{Cbasis1} and all identities $x^c$ where $c\neq d$.
\end{lemma}
\begin{proof}
Let $d\in\mathbb{Z}$ and let $H$ be the 1-dimensional Lie algebra over $K$. The algebra $H=\oplus_{i\mathbb{Z}}H_i$ is $\mathbb{Z}$-graded with $i$-th homogeneous component $H_i$ equal to $H$ if $i =d$ and 0 otherwise. It is clear that $H$ satisfies the graded identities \eqref{Cbasis1} as well as all graded identities $x^c$ where $c\neq d$ but does not satisfy the identity $x^d$. 
\end{proof}

\begin{corollary}
	The set of polynomials $\{x_d\mid d \leq -2\}$ is an independent set of graded identities in $L\langle X_\mathbb{Z}\rangle$.
\end{corollary}

\begin{theorem}
The graded identities $x^c$ ($c\leq -2$) and $[x^a_1, x^b_2]$,
where $a$ and $b$ are of the same parity, with $0\le a\le b$, form a minimal basis for the $\mathbb{Z}$-graded identities of $W_1$ over a field $K$ of characteristic 2.
\end{theorem}

The next corollary is a direct consequence of the above theorem together with Theorem \ref{bseminimalU_1}.
\begin{corollary}
	Over a field of characteristic two, the $\mathbb{Z}$-graded identities for the Lie algebra $U_1$, as well as $W_1$, do not admit any finite bases.
\end{corollary}

\end{document}